\documentclass[a4paper,11pt]{amsart}
\usepackage[T1]{fontenc}
\usepackage[utf8]{inputenc}
\usepackage{xifthen, xfrac, mathrsfs, url, amssymb} 
\usepackage[colorinlistoftodos]{todonotes}
\newtheorem{thm}{Theorem}[section]
\newtheorem{obs}[thm]{Observation}
\newtheorem{lem}[thm]{Lemma}

\newtheorem{alg}{Algorithm}

\theoremstyle{remark}

\newenvironment{poc}[1][]{\begin{proof}[\ifthenelse{\equal{#1}{}}{Proof of correctness}{Proof of correctness of #1}]}{\end{proof}}
\newcommand{\q}{q}

\newcommand{\gb}{\mathfrak{b}}
\newcommand{\gd}{\mathfrak{d}}

\newcommand{\gp}{\mathfrak{p}}
\newcommand{\gq}{\mathfrak{q}}
\newcommand{\gr}{\mathfrak{r}}
\newcommand{\CC}{\mathbb{C}}
\newcommand{\FF}{\mathbb{F}}

\newcommand{\QQ}{\mathbb{Q}}
\newcommand{\RR}{\mathbb{R}}
\newcommand{\UU}{\mathbb{U}}
\newcommand{\ZZ}{\mathbb{Z}}
\newcommand{\GB}{\mathfrak{B}}
\newcommand{\GP}{\mathfrak{P}}

\newcommand{\SB}{\mathscr{B}}

\newcommand{\classS}{C_{\setS}}
\newcommand{\ClassS}{\sfrac{\classS}{\classS^2}}
\newcommand{\K}{K}
\newcommand{\Kp}{\K_\gp}
\newcommand{\setS}{S}
\newcommand{\US}{U_\setS}
\newcommand{\UUS}{\UU_\setS}
\newcommand{\Sprime}{\setS'}

\newcommand{\into}{\rightarrowtail}

\newcommand{\units}[1]{#1^\times}
\newcommand{\squares}[1]{#1^{\times2}}
\newcommand{\sing}[1]{E_{#1}}
\newcommand{\singS}{\sing{\setS}}
\newcommand{\Sing}[1]{\mathbb{E}_{#1}}

\newcommand{\SingS}{\Sing{\setS}}
\newcommand{\st}{\mathrel{\mid}}

\DeclareMathOperator{\rank}{rk}
\newcommand{\rk}[1][]{\ifthenelse{\equal{#1}{}}{\rank_{2}}{\rank_{#1}}}
\newcommand{\class}[2][]{\ifthenelse{\equal{#1}{}}{[#2]}{[#2]_{#1}}}   

\newcommand{\un}[1][]{\ifthenelse{\equal{#1}{}}{\units{\K}}{\units{#1}}}
\newcommand{\sqgd}[1][]{\ifthenelse{\equal{#1}{}}{\sfrac{\un}{\squares{\K}}}{\sfrac{\units{#1}}{\squares{#1}}}}
\newcommand{\term}[1]{\emph{#1}}

\newcommand{\even}{\equiv 0\pmod{2}}
\newcommand{\odd}{\equiv 1\pmod{2}}

\newcommand{\form}[1]{\langle#1\rangle}

\DeclareMathOperator{\adim}{dim_a}
\DeclareMathOperator{\disc}{disc}
\DeclareMathOperator{\ord}{ord}
\DeclareMathOperator{\sgn}{sgn}
\def\clap#1{\hbox to 0pt{\hss#1\hss}}

\def\mathrlap{\mathpalette\mathrlapinternal}

\def\mathrlapinternal#1#2{\rlap{$\mathsurround=0pt#1{#2}$}}

\author[P. Koprowski \and B. Rothkegel]{Przemys{\l}aw Koprowski \and Beata Rothkegel}
\email{przemyslaw.koprowski@us.edu.pl}
\email{beata.rothkegel@us.edu.pl}
\title{The anisotropic part of a quadratic form over a number field}
\begin{document}
\maketitle
\begin{abstract}
It is well known that every non-degenerate quadratic form admits a decomposition into an orthogonal sum of its anisotropic part and a hyperbolic form. This decomposition is unique up to isometry. In this paper we present an algorithm for constructing an anisotropic part of a given form with coefficients in an arbitrary number field.
\end{abstract}

\section{Introduction}
The notion of isotropy is central to the theory of quadratic forms. Recall that a form~$\q$ is called \term{isotropic}, if there is a non-zero vector~$v$ such that $\q(v) = 0$. In geometric terms, this means that $v$ is self-orthogonal with respect to~$\q$. The celebrated Witt decomposition theorem (see e.g., \cite[Chapter~12]{Szymiczek1997}) says that every non-degenerate quadratic form~$\q$ is isometric to an orthogonal sum
\[
\q_a\perp w\times \form{1,-1},
\]
for some anisotropic form~$\q_a$, called an \term{anisotropic part} of~$\q$, and some $w\geq 0$, called the \term{Witt index} of~$\q$. The anisotropic part of~$\q$ is determined uniquely up to isometry. Its dimension is called the \term{anisotropic dimension} of~$\q$ and denoted $\adim(\q)$, hereafter.

From the computational point of view, given a non-degenerate quadratic form~$\q$, the following four problems arise immediately: 
\begin{enumerate}
\item\label{it:is_iso} determine whether $\q$ is isotropic or not;
\item\label{it:dim_a} compute the anisotropic dimension of~$\q$;
\item\label{it:q_a} construct its anisotropic part; 
\item\label{it:v_iso} if $\q$ is isotropic, find an isotropic vector.
\end{enumerate}
The problems above are listed in increasing order of difficulty. Indeed, if one can solve~\eqref{it:dim_a}, then it suffices to compare $\dim\q$ with $\adim\q$ to solve~\eqref{it:is_iso}. If one can construct an anisotropic part~$\q_a$ of~$\q$, then $\adim\q = \dim\q_a$. Finally, if one can solve~\eqref{it:v_iso}, then one may find on anisotropic part of~$\q$ removing successive hyperbolic planes till only an anisotropic form is left.

It is not at all surprising that most effort have been focused on forms over the rational. For such forms solutions to problem~\eqref{it:v_iso} have a long history dating back to Lagrange. Efficient algorithms for this task were devised by Cremona and Rusin in \cite{CreRus2003}, Simon in \cite{Simon2005} and Castel in \cite{Castel2013}. More recently, Quertier in \cite{Quertier2016} presented a method for finding a vector such is simultaneously isotropic with respect to two forms (of dimension at least~$13$) with rational coefficients. For forms over $\RR(x)$ task~\eqref{it:is_iso} was solved by the first author in \cite{Koprowski2008}, while the task~\eqref{it:q_a}, and consequently also \eqref{it:v_iso}, was proved to be unsolvable in general. Nonetheless, for  forms of dimension~$3$ there is a solution even to problem~\eqref{it:v_iso}, which is due to Schicho (see \cite{Schicho1998}). To some extend it resembles Lagrange approach for forms over~$\QQ$. Schicho's method was subsequently generalized by van~Hoeij and Cremona \cite{vHC2006} to forms with coefficients in multivariate rational function fields over either a finite field or the rationals. For forms with coefficients in real algebraic fields, tasks~\eqref{it:is_iso} and~\eqref{it:dim_a}  are solved in \cite{JK2016}. Algorithmic solutions to~\eqref{it:is_iso} and~\eqref{it:dim_a} for forms over number fields are given in \cite{KCz2018}. Analogous algorithms for global function fields have been recently invented by Darkey-Mensah (see \cite{DarkeyMensah2021}). Also recently, a solution to task~\eqref{it:q_a} has been found in \cite{DMKR2021} by the present authors and Darkey-Mensah. The goal of this paper is to present an algorithm that solves problem~\eqref{it:q_a} over an arbitrary number field. For forms of anisotropic dimension~$2$, the proposed algorithm is a generalization of the algorithm in \cite{DMKR2021}. For forms of anisotropic dimension~$3$ or more (except anisotropic dimension~$4$ over non-real fields) the algorithms presented in this paper are completely new. All the described algorithms were implemented in Magma package CQF \cite{Koprowski2020}. 

This paper is organized as follows. In Section~\ref{sec:notation} we establish the notation used throughout the paper and recall most of the relevant terminology. The two subsequent sections describe some auxiliary algorithms used in the main part. Section~\ref{sec:singular} presents an algorithm for constructing the group of $S$-singular elements (modulo squares) of a number field. Next, in Section~\ref{sec:signs} we discuss methods for finding elements that have prescribed signs with respect to different orderings of the field. The main result of this paper is the algorithm for constructing an anisotropic part of a given form. It is described in Section~\ref{sec:anisotropic_part}, which is divided into three subsections, dealing with different anisotropic dimensions. Finally, in Section~\ref{sec:example} we show an explicit example of how the algorithm works in practice.

Our algorithm for constructing an anisotropic part of a quadratic form depends on a number of auxiliary procedures. We assume that, beside the basic linear algebra routines, we have at our disposal the following tools from the arsenal of the computational algebraic number theory:
\begin{itemize}
\item An algorithm that checks if an ideal is principal, and if so finds its generator (see e.g., \cite[Section~6.5.5]{Cohen1993}).
\item Factorization of an ideal into prime ideals. There is a vast bibliography concerning this problem, see e.g., \cite[Algorithm 2.3.22]{Cohen2000} or \cite[\S2.2]{GMN2013}. In fact we will need only to factorize principal ideals.
\item A method for isolating real roots of a polynomial. There are probably hundreds of known techniques that can be used here. See for example \cite{AS05,AV10,BPR03}.
\item Closely connected to the previous point, are algorithms for enumerating all the real embeddings of a number field, see e.g., \cite[\S3.1]{Belabas2004}.
\item A method for computing the $\setS$-class group for some finite set~$\setS$, see for example \cite[Algorithm~7.4.6]{Cohen2000} and \cite{CDD1997, Biasse2014, Biasse2014a}
\item A related problem is the construction of the group of $\setS$-units, see e.g. \cite[Algorithm~7.4.8]{Cohen2000} or \cite{CDD1997}. 
\item Computation of the anisotropic dimension (or equivalently of the Witt index) of a quadratic form. This is described in \cite[Algorithm~9]{KCz2018}.
\item An algorithm for computing Hilbert symbol is given in \cite[Algorithm~6.6]{Voight13}.
\end{itemize}
All these algorithm are implemented in existing computer algebra systems, like for instance Magma \cite{BCP1997} (see also \cite{GMN2010}).

\section{Notation}\label{sec:notation}
Throughout this paper we use the following notation conventions. $\K$~always denotes a number field, that is a finite extension of~$\QQ$. The set of all places (classes of valuations) of~$\K$ is denoted~$\Omega_K$. We use fraktur letters $\gp, \gq, \gr,\dotsc$ for non-archimedean places of~$\K$. If $\gp$ is such a place, then $\ord_\gp: \un \to \ZZ$ is the associated discrete valuation, $\K(\gp)$ is the residue field and $\Kp$ the completion of~$\K$ at~$\gp$. Recall (see e.g. \cite[Theorem~VI.2.2]{Lam2005}) that if $\gp$ is nondya­dic (i.e. it does not divide~$2$), then the square class group $\sqgd[\Kp]$ consists of four cosets represented by $1$, $u_\gp$, $\pi_\gp$ and $u_\gp\pi_\gp$, where $\pi_\gp$ is a $\gp$-uniformizer and $u_\gp$ satisfies the conditions:
\[
\ord_\gp u_\gp = 0
\qquad\text{and}\qquad u_\gp\notin \squares{\Kp}.
\]
For any two elements $a, b\in \un$, we write $(a,b)_\gp$ for the Hilbert symbol of~$a$ and~$b$ at~$\gp$ (see e. g., \cite[Chapter VI]{Lam2005}). If $\q = \form{a_1, \dotsc, a_n}$ is a quadratic form, then 
\[
s_\gp(\q) := \prod_{i<j} (a_i, a_j)_\gp
\]
is the Hasse invariant of~$\q$ at~$\gp$ (see e.g., \cite[Definition~V.3.17]{Lam2005}). 

Further, $\disc(\q)$ is the discriminant of~$\q$, that is (see eg., \cite[Definition~15.2.1]{Szymiczek1997}):
\[
\disc(\q) = (-1)^{\frac12n(n-1)}\cdot \prod_{i=1}^n a_i.
\]
Moreover, we denote
\[
\GP(\q) := \bigl\{
\gp\in \Omega_\K\st \ord_\gp(a_i)\text{ is odd for some }i\leq n\bigr\}
\]
the set of primes of~$\K$ where at least one of the coefficients has an odd valuation.

The set of similarity classes of non-degenerate forms, equipped with binary operations induced by orthogonal sum and tensor product, is called the Witt ring of~$\K$ and denoted $W\K$ (see e.g., \cite{Lam2005, Szymiczek1997} for further information). The ideal class group of~$\K$ is denoted~$C_\K$. If $\setS$ is any finite set of places of~$\K$, then $\classS$ is the associate $\setS$-class group.

\section{Group of singular elements}\label{sec:singular}
In this section we gather some results concerning the group of $\setS$-singular elements, modulo squares. Let $\setS$ be a finite set of places of~$\K$, containing all archimedean places. We say that an element $\alpha\in \un$ is \term{$\setS$-singular} if it has even valuation at every prime $\gp\notin \setS$. The set of all $\setS$-singular elements is denoted~$\singS$. Observe that $1\in \singS$ and for every $\alpha\in \singS$, we have $\alpha\cdot \squares{\K} \subset \singS$. Thus, the notion of $\setS$-singularity extends canonically to square classes of~$\K$. We denote $\SingS := \sfrac{\singS}{\squares{\K}}$. Therefore we have
\[
\SingS = \bigl\{ \alpha\in \sqgd\st \ord_\gp\alpha\text{ is even for all }\gp \notin \setS\bigr\}.
\]
Assume that $\setS$ contains all dyadic primes of~$\K$. Recall that an element $\alpha\in \un$ is called an \term{$\setS$-unit} if $\ord_\gp\alpha = 0$ for every $\gp\notin \setS$. The set of $\setS$-units is denoted~$\US$. It is clear that $\US\subset \singS$. The canonical embedding $\sfrac{\US}{\US^2}\into \sqgd$ lets us identify $\sfrac{\US}{\US^2}$ with a subset of~$\SingS$. We shall denote this subset by~$\UUS$.

The construction of the group~$\SingS$ of singular elements modulo squares is closely related to the computation of~$\UUS$. One method, sketched in Magma manual \cite{CBFS2015}, is to enlarge~$\setS$ to a new set $\Sprime$, such that the $\Sprime$-class number is odd. (It suffices to adjoin to~$\setS$ a set of primes whose classes form a basis of $\ClassS$.) Then, it is known that $\UU_{\Sprime} = \Sing{\Sprime}$, and so one can obtain~$\SingS$ as a $\FF_2$-subspace of~$\Sing{\Sprime}$.

Below we present an alternative approach, which is due to Alfred Czogała.

\begin{alg}\label{alg:S-singular}
Let $\setS$ be a finite set of places of a number field~$\K$, that contains all the infinite and dyadic places. This algorithm constructs a basis \textup(over $\FF_2$\textup) of the group~$\SingS$.
\begin{enumerate}
\item Let ${}_2\classS$ be the subgroup of the $\setS$-class group~$\classS$, consisting of elements of order $\leq 2$.
\item Find a set~$\GB$ of primes that form a basis of~${}_2\classS$.
\item For every $\gb\in \GB$ find an element $\lambda_\gb$ that generates the \textup(principal\textup) ideal~$\gb^2$.
\item Find a basis~$\SB$ of $\UUS$.
\item Output $\SB\cup \{\lambda_\gb\st \gb\in \GB\}$.
\end{enumerate}
\end{alg}

\begin{poc}
Consider a map $\psi : \singS\to {}_2\classS$ given by the formula:
\[
\psi(\alpha) := \Biggl[\;\prod_{\gp\notin \setS} \gp^{\sfrac{(\ord_\gp\alpha)}{2}}\Biggr]_{\mathrlap{\setS}}\;.
\]
It is clear that $\psi$ is a group epimorphism. Observe that the kernel of~$\psi$ coincides with $U_\setS\cdot \squares{\K}$. Indeed, suppose that $\psi(\alpha)$ vanishes for some $\alpha\in \un$. This means that the ideal
\[
\prod_{\gp\notin \setS} \gp^{\sfrac{(\ord_\gp\alpha)}{2}}
\]
is principal. Hence, there is $\beta\in \un$ such that $2\ord_\gp\beta = \ord_\gp\alpha$ for every $\gp\notin \setS$. Then $\sfrac{\alpha}{\beta^2}$ is an $\setS$-unit. This proves the inclusion $\ker\psi \subset U_\setS\cdot \squares{\K}$. To show the other inclusion, observe first that $U_\setS$ is trivially contained in $\ker\psi$ and if $\alpha\in \squares{\K}$, say $\alpha = \beta^2$ for some $\beta\in \un$, then $\psi(\alpha) = \class[S]{(\beta)} = 1$. This way, we have proved the claim.

Now, let $\GB = \{\gb_1, \dotsc, \gb_m\}$ be a basis of ${}_2\classS$. Then $\class[\setS]{\gb_i}^2$ is principal for every $i\leq m$. Hence, the corresponding generator $\lambda_{\gb_i}$ exists. It is clear that $\lambda_{\gb_i}\in \SingS$. Consider an exact sequence
\[
0\to \UUS \xrightarrow{i} \SingS \xrightarrow{\psi} {}_2\classS\to 0,
\]
where $i$ is the canonical inclusion. The sequence splits since all three groups are $\FF_2$-vector spaces. This shows that $\SB\cup \{\lambda_{\gb_1}, \dotsc, \lambda_{\gb_m}\}$ is a basis of~$\SingS$.
\end{poc}

\section{Elements of independent signs}\label{sec:signs}
When dealing with formally real number fields, we often need to construct elements of independent signs in distinct real embeddings of the number field. Here we shortly explain how to construct them. The method presented in Algorithm~\ref{alg:ordering_separation} below is not new, however, the authors are not aware of any easily accessible reference. Hence, for the reader's convenience, we provide an explicit pseudo-code.

\begin{alg}\label{alg:ordering_separation}
Given a number field $\K = \QQ(\theta)$ with $r$ real embeddings, denoted hereafter $\sigma_1, \dotsc, \sigma_r$, and a subset $I\subseteq \{1, \dotsc, r\}$, this algorithm returns an element $\rho\in \un$ such that $\sigma_i(\rho) < 0$ for $i\in I$ and $\sigma_j(\rho)> 0$ for $j\notin I$.
\begin{enumerate}
\item Let $f$ be the defining polynomial for~$\K$ and $\xi_1 := \sigma_1(\theta), \dotsc, \xi_r := \sigma_r(\theta) \in \RR$ be all the real roots of~$f$. 
\item Find intervals $(a_i, b_i)$ for $i\leq r$, with rational endpoints, that isolate roots of~$f$ i.e. $\xi_i\in (a_i, b_i)$ for every~$i$.
\item Set $\eta_i := (\theta - a_i)(\theta - b_i)$ for $i\leq r$.
\item Output $\rho := \prod\limits_{i\in I}\eta_i$.
\end{enumerate}
\end{alg}

The algorithm is simple enough that we take the liberty to omit a rigorous proof of its correctness. Let us only mention that in a practical implementation, the elements $\eta_1, \dotsc, \eta_r$ are, of course, constructed only once and cached between successive executions of this algorithm. Unfortunately this algorithm is not fully sufficient for applications we have in mind. In particular we need $\rho$ to be a local square at some fixed primes. This goal is achieved in Algorithm~\ref{alg:strong_ordering_separation} below, but first we need to introduce the following auxiliary procedure.

\begin{alg}\label{alg:positive_approximation}
Given a finite set $\setS = \{\gp_1,\dotsc, \gp_n\}$ of non-archimedean places, corresponding exponents $k_1, \dotsc, k_n$, and elements $\lambda_1, \dotsc, \lambda_n\in \un$, this algorithm constructs a \textbf{totally positive} element $\alpha\in \un$ such that
\[
\alpha\equiv \lambda_i\pmod{\gp_i^{k_i}}
\]
for every $i\leq n$.
\begin{enumerate}
\item Using Chinese Remainder Theorem construct $\beta\in \un$ such that 
\[
\beta\equiv \lambda_i\pmod{\gp_i^{k_i}}
\]
for every $i\leq n$.
\item Let $\Sprime$ be the set of all prime numbers dominated by elements in~$\setS$.
\item For every $p\in \Sprime$ set
\[
m(p) := \max\{ k_i\st \gp_i\text{ dominates }p,\ \gp_i\in\setS\}.
\]
\item Set
\[
s := \prod_{p\in \Sprime} p^{m(p)}.
\]
\item Find a positive integer~$t$ such that 
\[
t\cdot s > \max\{ \sigma_j(-\beta)\st j\leq r\}, 
\]
where $\sigma_1, \dotsc, \sigma_r : \K\into \RR$ are all the real embeddings of~$\K$.
\item Output $\alpha := \beta + t\cdot s$.
\end{enumerate}
\end{alg}

\begin{poc}
First, we prove that $\alpha$ is totally positive. Fix any real embedding $\sigma_j: \K\into \RR$. We have: 
\[
\sigma_j(\alpha)
= \sigma_j( \beta + t\cdot s)
= \sigma_j(\beta) + t\cdot s
> \sigma_j(\beta) + \sigma_j(-\beta)
= 0.
\]
Thus, $\alpha$ is indeed totally positive. Next, observe that for every $i\leq n$ we have $\ord_{\gp_i} s\geq k_i$, hence 
\[
\alpha \equiv \beta \equiv \lambda_i\pmod{ \gp_i^{k_i}}
\]
and this ends the proof.
\end{poc}

\begin{alg}\label{alg:strong_ordering_separation}
Let $\K = \QQ(\theta)$ be a formally real number field with $r$ real embeddings, denoted $\sigma_1, \dotsc, \sigma_r$, hereafter. Given a subset $I\subseteq \{1,\dotsc, r\}$ and a finite set~$\setS$ of non-archimedean places, this algorithm returns an element $\rho\in \un$ such that 
\begin{enumerate}\renewcommand{\theenumi}{$\roman{enumi}$}
\item\label{it:negative_in_I} $\sigma_i(\rho) < 0$ for $i\in I$,
\item\label{it:positive_out_I} $\sigma_i(\rho) > 0$ for $i\notin I$,
\end{enumerate}
and $\rho$ is a local square at every $\gp\in S$.
\begin{enumerate}
\item Construct an element $\alpha_1\in \un$ such that
\[
\sgn \sigma_i(\alpha_1) =
\begin{cases}
-1 &\text{if }i\in I\\
\phantom{-}1 &\text{if }i\notin I,
\end{cases}
\]
for every $i\leq r$.
\item Use Algorithm~\ref{alg:positive_approximation} to construct a totally positive element $\alpha_2\in \un$, that is congruent to~$\alpha_1$ modulo $\gp^{1 + \ord_\gp4}$ for every $\gp\in \setS$.
\item Output $\rho = \alpha_1\cdot \alpha_2$.
\end{enumerate}
\end{alg}

\begin{poc}
Since $\alpha_2$ is totally positive, it is clear that $\sgn \sigma_i(\rho) = \sgn\sigma_i(\alpha_1)$ satisfies conditions~\eqref{it:negative_in_I} and~\eqref{it:positive_out_I}. Moreover, for every prime $\gp\in\setS$, we have
\[
\rho = \alpha_1\cdot \alpha_2 \equiv \alpha_1^2\pmod{\gp^{1 + \ord_\gp 4}}.
\]
Thus, $\rho$ is a local square by the well known consequence of the Local Square Theorem (see e.g. \cite[Corollary~VI.2.20]{Lam2005}).
\end{poc}

\section{Computing the anisotropic part}\label{sec:anisotropic_part}
In this section we present our main algorithm that constructs an an\-iso\-tropic part of a given quadratic form. Except for forms of an anisotropic dimension one, that are trivial to handle (see Observation~\ref{obs:adim_1}), the general idea is to construct the anisotropic part incrementally. In each step we will drop the anisotropic dimension by one, till we obtain a form that has a binary anisotropic part. The different anisotropic dimensions are discussed in a separate subsections.

\subsection{Anisotropic dimension four and above}
\begin{alg}\label{alg:dim_4}
Given a quadratic form $\q = \form{a_1, \dotsc, a_n}$ of anisotropic dimension~$d\geq 4$, this algorithm constructs an element $\alpha\in \un$ such that $\adim\bigl(\q\perp \form{-\alpha}\bigr) = d - 1$.
\begin{enumerate}
\item\label{st:nonreal} If $\K$ is non-real, then output $1$ and quit. 
\item Let $\sigma_1, \dotsc, \sigma_r: \K\into \RR$ be all the real embeddings of~$\K$.
\item Set
\[
I_+ := \bigl\{ i\leq r\st \sgn \sigma_i(\q) = d \bigr\},\qquad
I_- := \bigl\{ i\leq r\st \sgn \sigma_i(\q) = -d \bigr\}.
\]
\item Construct on element $\alpha\in \un$ such that $\sigma_i(\alpha)> 0$ for all $i\in I_+$ and $\sigma_i(\alpha) < 0$ for all $i\in I_-$.
\item Output $\alpha$.
\end{enumerate}
\end{alg}

\begin{poc}
Let us begin with the case of non-real fields. It is well known that over a non-real global field every form of dimension $\geq 5$ is isotropic (see e.g., \cite[Corollary~VI.3.5]{Lam2005}). Hence, if this is the case, then $d$ must be $4$ and the form $\q\perp\form{-1}$ is isotropic. Since the parity of the dimension and anisotropic dimension coincide, we have $\adim\bigl(\q\perp\form{-1}\bigr) = d-1$. This proves the correctness of step~\eqref{st:nonreal}.

In what follows we assume that $\K$ is formally real. Over $\CC$ every form of dimension greater than~$1$ is always isotropic. Consequently, the local anisotropic dimension of $\q\perp \form{-\alpha}$ at any complex place cannot exceed~$1$, hence is trivially strictly smaller than~$d$. In turn, fix a real embedding $\sigma_i$ of~$\K$. If $i\in I_+\cup I_-$, then by the definition of~$\alpha$ we have $\bigl|\sgn \sigma_i \bigl(\q\perp \form{-\alpha}\bigr)\bigr| = d - 1$. Conversely suppose that $i \notin I_+\cup I_-$. Then $|\sgn \sigma_i( \q)| \leq d - 2$ and consequently $\bigl|\sgn \sigma_i\bigl( \q \perp \form{-\alpha} \bigr)\bigr| \leq d - 1$. Finally, take a completion $\Kp$ of~$\K$ at some finite prime~$\gp$. Every form of dimension $\geq 5$ over $\Kp$ is isotropic, hence $\adim \bigl((\q\perp \form{-\alpha})\otimes \Kp\bigr) \leq 4\leq d$ and if $d = 4$ the parity preservation implies $\adim \bigl((\q\perp \form{-\alpha})\otimes \Kp\bigr) \leq 3$. All in all, it follows from the local-global principle (see e.g., \cite[Setion~VI.3]{Lam2005}) that the anisotropic dimension of $\q\perp \form{-\alpha}$ is $d - 1$, as claimed. 
\end{poc}

\subsection{Anisotropic dimension three}
We will now deal with forms of an\-iso\-tro\-pic dimension three. As in the previous section, the idea is to find an element $\alpha\in \un$ such that by adding $\form{-\alpha}$ to~$\q$ we will further drop the anisotropic dimension. For clarity of exposition we will deal with real and non-real cases separately. We begin with non-real fields, where the situation is considerably simpler.

\begin{alg}\label{alg:adim3_nonreal}
Let~$\K$ be a non-real number field and~$\q$ be a quadratic form over~$\K$ of the anisotropic dimension~$3$. This algorithm construct an element $\alpha\in \un$ such that $\adim\bigl( \q\perp \form{-\alpha}\bigr) = 2$. 
\begin{enumerate}
\item Set $\setS := \GP(\q) \cup \{\gd_1, \dotsc, \gd_l\}$, where $\gd_1, \dotsc, \gd_l$ are all the dyadic primes of~$\K$.
\item Using Chinese Remainder Theorem find $\alpha\in \un$ such that
\[
\alpha \equiv
\left\{\begin{array}{lll}
\disc(\q) - 1 &\pmod{\gp}   &\text{if }\ord_\gp\disc(\q)\notin 2\ZZ\\
\pi_\gp       &\pmod{\gp^2} &\text{if }\ord_\gp\disc(\q)\in 2\ZZ
\end{array}\right.
\]
for every prime $\gp\in S$.
\item Output~$\alpha$.
\end{enumerate}
\end{alg}

\begin{poc}
Let $\q_a = \form{a, b, c}$ be the sought anisotropic part of~$\q$. We have
\[
\q \cong \form{a, b, c}\perp w\times \form{1,-1},
\]
where $w = w(\q)$ is the Witt index of~$\q$. Let $\alpha\in \un$ be the element constructed by the algorithm. We claim that $\q_a\perp \form{-\alpha}$ is isotropic. Take a prime $\gp \notin S\cup \GP(\q_a)$. Then $a$, $b$ and~$c$ have even valuations at~$\gp$, hence $\q_a\otimes \Kp$, being a ternary form, is isotropic (by \cite[Corollary~VI.2.5]{Lam2005}). Therefore $\adim( \q_a\otimes \Kp )$ equals~$1$, since it must have the same parity as the dimension of~$\q_a$. It follows that
\[
\adim\bigl( (\q_a\perp\form{-\alpha})\otimes \Kp\bigr) \in \{0, 2\}
\]
and so $\q_a\perp\form{-\alpha}$ is locally isotropic at~$\gp$.

Next, take a prime $\gp\in \GP(\q_a)\setminus \setS$. In particular $\gp$ is non-dyadic. We have $\ord_\gp\disc(\q) \even$, hence precisely two of the coefficients $a$, $b$ and $c$ must have an odd valuation at~$\gp$. Without loss of generality we may assume that these are $a$ and~$b$. The second residue homomorphism $W\Kp\to W\K(\gp)$ vanishes at $\q\otimes\Kp$ so it must vanish on $\q_a\otimes \Kp$, as well. It follows that $\form{a, b}\otimes\Kp$ is hyperbolic and consequently $\adim(\q_a\otimes \Kp) = 1$. This way we show that $\adim\bigl((\q_a\perp \form{-\alpha})\otimes \Kp\bigr) \leq 2$.

Finally, pick a prime $\gp\in \setS$. Be it dyadic or non-dyadic. It is well known (see e.g. \cite[Theorem~VI.2.10 and Corollary~VI.2.15]{Lam2005}) that $\form{1, -u, -\pi_\gp, u\pi_\gp}$ is a unique (up to isometry) anisotropic form over~$\K_\gp$. The determinant of this form is~$1$. In particular, in the square class group $\sqgd[\K_\gp]$, every coefficient of this form is a product of the other three. Suppose a contrario that $\q_a\perp \form{-\alpha}$ is anisotropic. Therefore
\[
\q_a\perp \form{-\alpha} = \form{a, b, c, -\alpha} \cong \form{1, -u, -\pi_\gp, u\pi_\gp},
\]
and so we have $\alpha \equiv -abc = \disc(\q)\pmod{\squares{\K_\gp}}$. Hence, $\alpha = x^2\cdot \disc(\q)$ for some $x\in \units{\Kp}$. Consider two cases. If $\ord_\gp \disc(\q)$ is odd, then $\alpha \equiv \disc(\q) - 1\pmod{\gp}$ and so it has even valuation. This is impossible since at the same time $\alpha = x^2\cdot \disc(\q)$ has odd valuation.

Conversely, suppose that $\ord_\gp \disc(\q)$ is even. Then $\alpha\equiv \pi_\gp\pmod{\gp^2}$ has odd valuation which contradicts the fact that $\alpha = \disc(\q)\cdot x^2$ must have even valuation. The contradiction follows from the supposition that $\form{a, b, c, -\alpha}\otimes \Kp$ can be anisotropic. 

We have shown that $\q_a\perp \form{-\alpha}$ is locally isotropic at every finite place of~$\K$. Since $\K$ is non-real, every infinite place of~$\K$ is complex and $\q_a\perp \form{-\alpha}$ is trivially locally isotropic at complex places. The local-global principle asserts that $\q_a\perp \form{-\alpha}$ is isotropic over~$\K$ and this implies that $\adim\bigl(\q\perp\form{-\alpha}\bigr) < 3$.
\end{poc}

We may now turn our attention to formally real fields. 

\begin{alg}\label{alg:adim3_real}
Given a quadratic form $\q = \form{a_1, \dotsc, a_n}$ of the anisotropic dimension $\adim(\q) = 3$ over a formally real number field~$\K$, this algorithm constructs an element $\alpha\in \un$ such that $\adim\bigl( \q\perp \form{-\alpha} \bigr) = 2$.
\begin{enumerate}
\item Set $\setS := \GP(\q) \cup \{\gd_1, \dotsc, \gd_l\}$, where $\gd_1, \dotsc, \gd_l$ are all the dyadic primes of~$\K$.
\item Let $\sigma_1, \dotsc, \sigma_r :\K\into \RR$ be all the real embeddings of~$\K$. Call Algorithm~\ref{alg:strong_ordering_separation} to find an element $\alpha_1\in \un$ such that
\begin{equation}\tag{\ensuremath{\spadesuit}}\label{eq:dim_3_signs}
\sgn \sigma_i(\alpha_1) =
\begin{cases}
+1, &\text{if }\sgn\sigma_i(\q) > 0,\\
-1, &\text{if }\sgn\sigma_i(\q) < 0
\end{cases}
\end{equation}
and $\alpha_1$ is a local square at every prime $\gp\in S$.
\item Using Algorithm~\ref{alg:positive_approximation} construct a totally positive element~$\alpha_2$ such that 
\[
\alpha_2 \equiv 
\left\{
\begin{array}{lll}
\disc(\q) - 1 &\pmod{\gp}  & \text{if } \ord_\gp\disc(\q)\odd\\
\pi_\gp      &\pmod{\gp^2} & \text{if } \ord_\gp\disc(\q)\even
\end{array}
\right.
\]
for every $\gp\in S$.
\item Output $\alpha := \alpha_1\cdot \alpha_2$.
\end{enumerate}
\end{alg}

\begin{poc}
We follow similar lines as in the proof of correctness of Algorithm~\ref{alg:adim3_nonreal}. We shall show that $\q_a \perp \form{-\alpha}$ is isotropic­, where $\q_a = \form{a, b, c}$ is the sought anisotropic part of~$\q$.

It is trivially isotropic at complex places. Now, take a real embedding~$\sigma_i$ of~$\K$. We know that $\alpha_2$ is totally positive. Thus, we infer from~\eqref{eq:dim_3_signs} that
\[
\sgn \sigma_i\bigl( \q_a\perp \form{-\alpha} \bigr)
= \sgn \sigma_i\bigl( \q\perp \form{-\alpha_1} \bigr)
\in \{ 0, \pm 2 \}
\]
and so $\sigma_i\bigl( \q_a\perp \form{-\alpha}\bigr)$ is indeed isotropic.

Now, it is time to turn our attention to non-archimedean places. For a prime~$\gp$ not in~$\setS$ the same arguments as used for Algorithm~\ref{alg:adim3_nonreal} show that $\adim\bigl(((\q_a\perp \form{-\alpha}) \otimes \Kp\bigr) \leq 2$. For primes~$\gp$ sitting in $\setS$ the arguments used in the abovementioned proof show that
\[
\adim\bigl((\q_a\perp \form{-\alpha_2})\otimes \Kp\bigr) \leq 2.
\]
But for these primes we have $\alpha_1\in \squares{\Kp}$, hence $\form{-\alpha} \cong \form{-\alpha_2}$. This means that $\q_a\perp \form{-\alpha}$ is locally isotropic at every place of~$\K$. Consequently it is isotropic over~$\K$ by the local-global principle. It follows that
\[
\adim\bigl( \q\perp\form{-\alpha}\bigr) \leq 2
\]
and this proves the correctness of the algorithm. 
\end{poc}

\subsection{Anisotropic dimensions one and two}
Once we managed to reduce the anisotropic dimension to two, it is time to explicitly construct an anisotropic part of a given form. This task is achieved by the following algorithm.

\begin{alg}\label{alg:adim2}
Given a quadratic form $\q = \form{a_1, \dotsc, a_n}$ of an anisotropic dimension~$2$, with coefficients in a number field~$\K$, this algorithm constructs an anisotropic part $\q_a$ of~$\q$. 
\begin{enumerate}
\item\label{st:even_Witt_inx} If the Witt index~$w(\q)$ of~$\q$ is not divisible by~$4$, then replace~$\q$ by $\q\perp w'\times \form{-1,1}$, where $w'+w(\q)\equiv 0\pmod{4}$.
\item Compute the discriminant $d := \disc\q$. 
\item\label{st:dyadic} Set $\setS := \GP(\q) \cup \{\gd_1, \dotsc, \gd_l\}$, where $\gd_1, \dotsc, \gd_l$ are all the dyadic primes of~$\K$.
\item\label{st:embeddings} Let $\sigma_1, \dotsc, \sigma_r : \K\into \RR$ be the real embeddings of~$\K$ such that $\sigma_i(d)$ is negative.
\item Repeat the following steps:
  \begin{enumerate}
  \item\label{st:basis} Construct a basis $\{\beta_1, \dotsc, \beta_m\}$ of the group $\SingS$ of $\setS$-singular elements modulo squares.
  \item For every real embedding $\sigma_i$ with $i\leq r$, set 
  \[
  v_i := 
  \begin{cases}
  1 & \text{if }\sgn \sigma_i(q) = -2,\\
  0 & \text{if }\sgn \sigma_i(q) = 2.
  \end{cases}
  \]
  \item Let $\setS = \{ \gp_1, \dotsc, \gp_s\}$. For $i\leq s$, set 
  \[
  w_i := 
  \begin{cases}
  1 & \text{if }s_{\gp_i}\q = -1,\\
  0 & \text{if }s_{\gp_i}\q = 1.
  \end{cases}
  \]
  Here $s_{\gp_i}\q$ is the Hasse invariant of $\q\otimes \K_{\gp_i}$.
  \item Construct a matrix $A = (a_{ij})$ with $r$ rows \textup(indexed by the real embeddings $\sigma_1, \dotsc, \sigma_r$\textup) and $m$ columns \textup(indexed by the elements of the basis of~$\SingS$ computed in step~\eqref{st:basis}\textup), setting 
  \[
  a_{ij} := 
  \begin{cases}
  1 & \text{if }\sigma_i(\beta_j) < 0,\\
  0 & \text{if }\sigma_i(\beta_j) > 0.
  \end{cases}
  \]
  \item Construct a matrix $B = (b_{ij})$ with $s$ rows \textup(indexed by the primes $\gp_1, \dotsc, \gp_s$ in~$\setS$\textup) and $m$ columns, setting 
  \[
  b_{ij} := 
  \begin{cases}
  1 & \text{if }(\beta_j, d)_{\gp_i} = -1,\\
  0 & \text{if }(\beta_j, d)_{\gp_i} = 1.
  \end{cases}
  \]
  Here $(\beta_j, d)_{\gp_i}$ is the $\gp_i$-adic  Hilbert symbol. 
  \item\label{st:adim_2:alpha} If the following system of $\FF_2$-linear equations
  \begin{equation}\tag{\ensuremath{\clubsuit}}\label{eq:dima_2}
  \left(\begin{array}{c}
  A \\\hline B
  \end{array}\right)\cdot
  \begin{pmatrix} \varepsilon_1\\\vdots\\\varepsilon_m\end{pmatrix} =
  \left(\begin{array}{c} v_1\\\vdots\\ v_r\\\hline w_1\\\vdots\\ w_s\end{array}\right)
  \end{equation}
  has a solution, then set 
  \[
  \alpha := \prod_{i=1}^m\beta_i^{\varepsilon_i}
  \]
  and exit the loop.
  \item Otherwise find a new non-archimedean place $\gq\notin \setS$, append it to~$\setS$, and reiterate the loop. 
  \end{enumerate}
\item Output $q_a := \form{\alpha, -\alpha\cdot d}$.
\end{enumerate}
\end{alg}


The proof of correctness of the algorithm needs to be preceded by  two lemmas.

\begin{lem}\label{lem:big_set}
Let $\q$ be a quadratic form over a number field~$\K$ and $\gd_1, \dotsc, \gd_l$ be all the dyadic places of~$\K$. If $\adim(\q) = 2$, then there is a non-archimedean place~$\gq$ such that among all \textup(necessarily isometric\textup) anisotropic parts of~$\q$, there is at least one, denoted~$\q_a$ hereafter, satisfying the condition 
\[
\GP(\q_a) \subseteq \GP(\q)\cup \{\gd_1, \dotsc, \gd_l\}\cup \{\gq\}.
\]
\end{lem}

\begin{proof}
As in step~\eqref{st:dyadic} of the algorithm, denote $\setS := \GP(\q)\cup \{\gd_1, \dotsc, \gd_l\}$. By assumption, $\adim(\q) = 2$, hence there is $\alpha\in \un$ such that
\begin{equation}\label{eq:adim_2_alpha}
q\cong \form{\alpha, -\alpha d}\perp w\times \form{1,-1},
\end{equation}
where $d:= \disc(\q)$ is the discriminant and $w := w(\q)$ is the Witt index of~$\q$. We will show that $\alpha$ can be selected to be $\bigl(S\cup \{\gq\}\bigr)$-singular for some place~$\gq$.

Fix any~$\alpha$ that satisfies condition~\eqref{eq:adim_2_alpha}. It follows from \cite[Lemma~2.1]{LW1992} that there is a place $\gq\notin S$ and an element $\gamma\in \un$ such that:
\begin{enumerate}\renewcommand{\theenumi}{$C_{\arabic{enumi}}$}
\item\label{it:reals} $\sgn\sigma(\gamma) = \sgn\sigma(\alpha)$ for every real embedding $\sigma$ of~$\K$;
\item\label{it:S_nondyadic} $\gamma\equiv \alpha\pmod{\gp}$ for every non-dyadic prime $\gp\in S$;
\item\label{it_S_dyadic} $\gamma\equiv \alpha\pmod{\gd_i^{1 + \ord_{\gd_i}4}}$ for every dyadic prime $\gd_i$, $i \leq l$;
\item $\ord_\gq\gamma = 1$;
\item $\ord_\gr\gamma = 0$ for every prime $\gr\notin S\cup \{\gq\}$.
\end{enumerate}
It is clear that $\gamma$ is $\bigl(S \cup \{\gq\}\bigr)$-singular. We claim that the following isometry holds:
\[
\form{\gamma, -\gamma d} \cong \form{\alpha, -\alpha d}.
\]
It holds locally at every archimedean place---indeed, for complex places this is trivial and for the real ones it follows from~\eqref{it:reals}. Consider now a prime $\gp\in S$, either dyadic or non-dyadic. By (\ref{it:S_nondyadic}/\ref{it_S_dyadic}) and the Local Square Theorem, we obtain $\gamma\cdot \squares{\Kp} = \alpha\cdot \squares{\Kp}$ and so $\form{\gamma, -\gamma d}\otimes \Kp\cong \form{\alpha, -\alpha d}\otimes \Kp$. Conversely, take a prime~$\gp$ not in~$\setS$ but distinct from~$\gq$. Then
\[
\ord_\gp d \equiv \ord_\gp\gamma = 0 \pmod{2}. 
\]
We need to consider two cases. If $\ord_\gp\alpha$ is also even, then 
\[
\form{\alpha, -\alpha d}\otimes \Kp\cong 
\form{1, -d}\otimes \Kp\cong 
\form{\gamma, -\gamma d}\otimes \Kp.
\]
On the other hand, if $\ord_\gp\alpha$ is odd, we consider the second residue homomorphism $W\Kp\to W\K(\gp)$. The forms~$\q$ and~$\q_a$ are similar, hence they map to the same class in $W\K(\gp)$, but $\gp\notin \GP(\q)$, thus $\q$ is mapped to the null element of $W\K(\gp)$. Consequently $\form{\alpha, -\alpha d}\otimes \Kp$ is hyperbolic and so is $\form{\gamma, -\gamma d}\otimes \Kp$. Therefore, the two forms are again isometric. Finally, we consider the localization of~$\K$ at the singled out place~$\gq$. From the previous part we obtain that the Hilbert symbols $(\alpha, d)_\gp$ and $(\gamma, d)_\gp$ coincide for  every prime $\gp\ne \gq$. The Hilbert reciprocity law implies that $(\alpha, d)_\gq = (\gamma, d)_\gq$, as well. This way we have proved that $\form{\alpha, -\alpha d}$ and $\form{\gamma, -\gamma d}$ are locally isometric at every place of~$\K$, hence they are isometric over~$\K$ by the local-global principle. This proves the claim. It follows that, the form $\form{\gamma, -\gamma d}$ is the anisotropic part of~$\q$ that we are looking for.
\end{proof}

\begin{lem}\label{lem:sgn_vs_disc}
Let $\K$ be a real closed field, $a, b\in \un$ and let $w$ be a positive integer. The form $\q = \form{a,b}\perp w\times \form{1, -1}$ is hyperbolic if and only if its discriminant is positive.
\end{lem}

\begin{proof} 
The discriminant of~$\q$ is $\disc(\q) = -ab$. If the form is hyperbolic, then $\disc(\q)$ is a square, hence it is positive. Conversely if $\disc(\q) > 0$, then $a$ and $b$ have opposite signs. Thus, $\q$ is hyperbolic.
\end{proof}

%

We are now in a position to prove correctness of the presented algorithm.

\begin{poc}[Algorithm~\ref{alg:adim2}]
Lemma~\ref{lem:big_set} asserts that there exists an anisotropic part~$\q_a$ of~$\q$ whose coefficients are $\bigl(\GP(\q)\cup \{\gd_1, \dotsc, \gd_l\}\cup \{\gq\}\bigr)$-singular for some prime~$\gq$ of~$\K$. This implies that the algorithm terminates. All we need to prove is that it outputs a correct result. Let
\[
\alpha := \beta_1^{\varepsilon_1}\dotsm \beta_m^{\varepsilon_m}
\]
be the element constructed in step~\eqref{st:adim_2:alpha}. We shall show that $\q$ and $\q_a = \form{\alpha, -\alpha\cdot d}$ are locally similar at every place of~$\K$. This is trivial for complex places. Consider a real embedding $\sigma :\K\into \RR$. First assume that $\sigma(d) > 0$ so $\sigma$ is not one of the embeddings we consider in step~\eqref{st:embeddings}. Then, obviously $\sigma(\q_a)$ is hyperbolic. Lemma~\ref{lem:sgn_vs_disc} says that $\sigma(\q)$ is hyperbolic, as well. Thus, the two forms are similar, as claimed.

Conversely assume that $\sigma = \sigma_i$ is one of the embeddings in step~\eqref{st:embeddings}, hence $\sigma_i(d) < 0$. We have
\begin{align*}
\sgn\sigma_i(\alpha)
&= \prod_{j = 1}^m \sgn\sigma_i(\beta_j^{\varepsilon_j}\bigr)\\
&= (-1)^{a_{i1}\varepsilon_1}\dotsm (-1)^{a_{im}\varepsilon_m}\\
&= (-1)^{v_i}\\
&= \frac12\sgn \sigma_i(\q),
\end{align*}
since $\varepsilon_1, \dotsc, \varepsilon_m$ form a solution to the system~\eqref{eq:dima_2}.

We now turn our attention to finite places of~$\K$. First, fix a prime $\gp\notin \setS$. Then $\alpha$ as well as all the coefficients $a_1, \dotsc, a_n$ of~$\q$ (consequently $d$, too) have even valuations at~$\gp$. It follows from \cite[Corollary~V1.2.5]{Lam2005}, and the very definition of the Hasse invariant, that both the Hasse invariants $s_\gp\q$ and $s_\gp\q_a$ vanish. We constructed~$\q_a$ in such a manner that the discriminants of~$\q$ and~$\q_a$ coincide. It follows from \cite[Theorem~ V.3.21]{Lam2005} that $\q\otimes\Kp$ and $\q_a\otimes \Kp$ are similar. 

Finally, fix a non-archimedean place $\gp_i \in \setS$ with $i\leq s$. It can be either dyadic or non-dyadic. We have
\begin{align*}
s_{\gp_i}\q_a
&= s_{\gp_i}\form{\alpha, - \alpha\cdot d} \\
&= (\alpha, - \alpha\cdot d)_{\gp_i} \\
&= (\alpha, d)_{\gp_i}  \\
&= \bigl( \beta_1^{\varepsilon_1}\dotsm \beta_m^{\varepsilon_m}, d\bigr)_{\gp_i} \\
&= \prod_{\substack{j\leq m\\ (\beta_j, d)_{\gp_i} = -1}} (-1)^{\varepsilon_j}\\
&= (-1)^{b_{i1}\varepsilon_1}\dotsm (-1)^{b_{im}\varepsilon_m} \\
&= (-1)^{w_i}
= s_\gp\q.
\end{align*}
Thus, the same argument as in the previous part shows that $\q\otimes\Kp$ and $\q_a\otimes\Kp$ are similar. Notice that the fact that the Witt index of~$\q$ is divisible by~$4$, ensures that the Hasse invariants coincide also for dyadic primes.

All in all, $\q$ and $\q_a$ are locally similar everywhere. Consequently they are similar over~$\K$ by the local-global principle. The fact that the anisotropic dimension of~$\q$ is~$2$ implies that $\q_a$ is the sought anisotropic part of~$\q$.
\end{poc}

For the sake of completeness we should also discuss forms of the anisotropic dimension equal one. This case, however, is completely trivial.

\begin{obs}\label{obs:adim_1}
If the anisotropic dimension of~$\q$ is~$1$, then $\form{\disc(\q)}$ is an anisotropic part of~$\q$.
\end{obs}

\section{Example}\label{sec:example}
Below we present a simple example illustrating how the algorithms described in this paper work. Take $\K = \QQ(\sqrt{-7})$ and the quadratic form
\begin{multline*}
\q := \Bigl\langle
  -3 - 9\sqrt{-7}, 
  -1, 
  -2 - 6\sqrt{-7}, 
  1 - \sqrt{-7}, \\
  -6 + 4\sqrt{-7}, 
  -3 + 2\sqrt{-7}, 
  4 - 4\sqrt{-7}
  \Bigr\rangle
\end{multline*}
The anisotropic dimension of~$\q$ equals~$3$ and the discriminant is $\disc\q = - 61056 - 342912\sqrt{-7}$. There are precisely four primes that matter for~$\q$, including the two dyadic primes of~$\K$. These are:
\begin{align*}
\gd_1 &= \bigl(2, \tfrac12(1 + \sqrt{-7}) \bigr), &
\gd_2 &= \bigl(2, \tfrac12(7 + \sqrt{-7}) \bigr)\\
\gp_1 &= (3), &
\gp_2 &= \bigl(37, \tfrac12(7 + \sqrt{-7})\bigr)
\end{align*}
Using Algorithm~\ref{alg:adim3_nonreal}, we find that $\adim\bigl( \q\perp\form{-1406} \bigr) = 2$, since
\[
1406\equiv \disc\q - 1\pmod{\gp_1}
\qquad\text{and}\qquad
1406\equiv \pi_\gp\pmod{\gp^2}
\]
for $\gp\in \{\gd_1, \gd_2,\gp_2\}$. 

Subsequently, we apply Algorithm~\ref{alg:adim2} to the form $\q' := \q\perp\form{-1406}$. There are two new primes that originate from the inclusion of $-1406$. These primes are
\[
\gp_3 = (19), \qquad
\gp_4 = \bigl(37, 20 + \sqrt{-7} \bigr).
\]
Moreover, for system~\eqref{eq:dima_2} to become solvable, we append three more primes to the set~$\setS$. The additional primes are:
\[
\gp_5 = (5),\qquad 
\gp_6 = \bigl( \sqrt{-7} \bigr),\qquad
\gp_7 = \bigl(11, 2 + \sqrt{-7}\bigr).
\]
The system~\eqref{eq:dima_2} becomes:
\[
\begin{pmatrix}
0 & 0 & 1 & 0 & 1 & 1 & 1 & 1 & 1 & 1\\
0 & 1 & 1 & 0 & 0 & 1 & 1 & 0 & 0 & 0\\
0 & 0 & 1 & 1 & 1 & 0 & 1 & 1 & 1 & 1\\
0 & 1 & 1 & 0 & 1 & 1 & 0 & 1 & 1 & 1\\
0 & 1 & 0 & 0 & 0 & 0 & 1 & 1 & 0 & 0\\
0 & 1 & 0 & 1 & 1 & 1 & 0 & 0 & 1 & 1\\
0 & 0 & 0 & 0 & 0 & 0 & 0 & 0 & 0 & 0\\
0 & 0 & 0 & 0 & 0 & 0 & 0 & 0 & 0 & 0\\
0 & 0 & 0 & 0 & 0 & 0 & 0 & 0 & 0 & 0\\
\end{pmatrix}
\cdot 
\begin{pmatrix}
\varepsilon_1\\
\varepsilon_2\\
\varepsilon_3\\
\varepsilon_4\\
\varepsilon_5\\
\varepsilon_6\\
\varepsilon_7\\
\varepsilon_8\\
\varepsilon_9\\
\varepsilon_{10}\\
\end{pmatrix}
=
\begin{pmatrix}
0\\ 0\\ 0\\ 1\\ 0\\ 1\\ 0\\ 0\\ 0
\end{pmatrix}
\]
Then, $V = (0, 1, 1, 0, 0, 0, 0, 1, 0, 0)$ is a solution. This solution corresponds to $\alpha = \frac12\bigl(-27 - 19\sqrt{-7}\bigr)$. Therefore, $\q'_a = \form{\alpha, -\alpha\cdot \disc\q'}$ is an anisotropic part of the form~$\q'$. Consequently,
\begin{align*}
\q_a 
&= \form{1406}\perp \q'_a\\ 
&= \Bigl\langle 1406, \tfrac12(-27 - 19\sqrt{-7}), 30903025152 - 7324337664\sqrt{-7} \Bigr\rangle
\end{align*}
is the sought anisotropic part of~$\q$.

\section{Conclusion}
In this paper we present an explicit method for constructing an anisotropic part of a given quadratic form over a number field. The algorithm described in this paper have been implemented in CQF package \cite{Koprowski2020} for the computer algebra system Magma \cite{BCP1997}. This let us verify how the algorithms behave in practice. In order to test the efficiency of our solution we prepared two test-suits, both consisting of $20$ randomly generated quadratic forms. In the first test-suit we used $20$ forms, each of dimension~$20$, over a non-real field of degree~$20$ \cite[Number field 20.0.569468379011812486801.1]{lmfdb}. In the other test-suit we used $20$ forms of dimension~$10$ over a formally real field \cite[Number field 10.10.80803005003125.1]{lmfdb}, which has $10$ distinct orderings. Both test-suits were executed on a budget PC (Intel i5-9400F, $2.9$ GHz with 32GB of RAM). The computation times for the first test varied from~$276$ to~$365$ second, with the mean value of~$324$ seconds. The figures for the second test-suits were: $0.24$s, $3.42$s and $2.82$s, respectively. This shows that the presented method works in practice.

\subsubsection*{Acknowledgments.} We wish to thank Alfred Czogała, who showed us Algorithm~\ref{alg:S-singular} and allowed us include it in this paper.

\bibliographystyle{plain} 
\bibliography{anisotropic}

\begin{thebibliography}{10}

\bibitem{AS05}
A.~G. Akritas and A.~W. Strzebo{\'n}ski.
\newblock A comparative study of two real root isolation methods.
\newblock {\em Nonlinear Anal. Model. Control}, 10(4):297--304, 2005.

\bibitem{AV10}
Alkiviadis~G. Akritas and Panagiotis~S. Vigklas.
\newblock Counting the number of real roots in an interval with {V}incent's
  theorem.
\newblock {\em Bull. Math. Soc. Sci. Math. Roumanie (N.S.)},
  53(101)(3):201--211, 2010.

\bibitem{BPR03}
Saugata Basu, Richard Pollack, and Marie-Fran{\c{c}}oise Roy.
\newblock {\em Algorithms in real algebraic geometry}, volume~10 of {\em
  Algorithms and Computation in Mathematics}.
\newblock Springer-Verlag, Berlin, 2003.

\bibitem{Belabas2004}
Karim Belabas.
\newblock Topics in computational algebraic number theory.
\newblock {\em J. Th\'{e}or. Nombres Bordeaux}, 16(1):19--63, 2004.

\bibitem{Biasse2014}
Jean-Fran\c{c}ois Biasse.
\newblock An {$L(1/3)$} algorithm for ideal class group and regulator
  computation in certain number fields.
\newblock {\em Math. Comp.}, 83(288):2005--2031, 2014.

\bibitem{Biasse2014a}
Jean-Fran\c{c}ois Biasse.
\newblock Subexponential time relations in the class group of large degree
  number fields.
\newblock {\em Adv. Math. Commun.}, 8(4):407--425, 2014.

\bibitem{BCP1997}
Wieb Bosma, John Cannon, and Catherine Playoust.
\newblock The {M}agma algebra system. {I}. {T}he user language.
\newblock {\em J. Symbolic Comput.}, 24(3-4):235--265, 1997.
\newblock Computational algebra and number theory (London, 1993).

\bibitem{CBFS2015}
John Cannon, Wieb Bosma, Claus Fieker, and Allan~Steel (eds.).
\newblock {\em Handbook of Magma Functions}, 2.21 edition, 2015.

\bibitem{Castel2013}
Pierre Castel.
\newblock Solving quadratic equations in dimension 5 or more without factoring.
\newblock In {\em A{NTS} {X}---{P}roceedings of the {T}enth {A}lgorithmic
  {N}umber {T}heory {S}ymposium}, volume~1 of {\em Open Book Ser.}, pages
  213--233. Math. Sci. Publ., Berkeley, CA, 2013.

\bibitem{CDD1997}
H.~Cohen, F.~Diaz~y Diaz, and M.~Olivier.
\newblock Subexponential algorithms for class group and unit computations.
\newblock volume~24, pages 433--441. 1997.
\newblock Computational algebra and number theory (London, 1993).

\bibitem{Cohen1993}
Henri Cohen.
\newblock {\em A course in computational algebraic number theory}, volume 138
  of {\em Graduate Texts in Mathematics}.
\newblock Springer-Verlag, Berlin, 1993.

\bibitem{Cohen2000}
Henri Cohen.
\newblock {\em Advanced topics in computational number theory}, volume 193 of
  {\em Graduate Texts in Mathematics}.
\newblock Springer-Verlag, New York, 2000.

\bibitem{CreRus2003}
J.~E. Cremona and D.~Rusin.
\newblock Efficient solution of rational conics.
\newblock {\em Math. Comp.}, 72(243):1417--1441 (electronic), 2003.

\bibitem{DarkeyMensah2021}
Mawunyo~Kofi Darkey-Mensah.
\newblock Algorithms for quadratic forms over global function fields of odd
  characteristic, 2021.
\newblock preprint \url{https://arxiv.org/abs/2104.10547}, submitted for
  ISSAC'21 short communications.

\bibitem{DMKR2021}
Mawunyo~Kofi Darkey-Mensah, Przemys\l{}aw Koprowski, and Beata Rothkegel.
\newblock The anisotropic part of a quadratic form over a global function
  field.
\newblock In {\em Proceedings of the 2021 on International Symposium on
  Symbolic and Algebraic Computation}, ISSAC '21, page 115–122, New York, NY,
  USA, 2021. Association for Computing Machinery.

\bibitem{GMN2010}
J.~Guardia, J.~Montes, and E.~Nart.
\newblock Arithmetic in big number fields: the '+ideals' package, 2010.

\bibitem{GMN2013}
Jordi Gu\`ardia, Jes\'{u}s Montes, and Enric Nart.
\newblock A new computational approach to ideal theory in number fields.
\newblock {\em Found. Comput. Math.}, 13(5):729--762, 2013.

\bibitem{JK2016}
Konrad Jałowiecki and Przemysław Koprowski.
\newblock Algorithms for quadratic forms over real function fields.
\newblock In {\em Algebra, logic and number theory}, volume 108 of {\em Banach
  Center Publ.}, pages 133--141. Polish Acad. Sci. Inst. Math., Warsaw, 2016.

\bibitem{Koprowski2020}
Przemys{\l}aw Koprowski.
\newblock C{QF} {M}agma package.
\newblock {\em ACM Commun. Comput. Algebra}, 54(2):53--56, 2020.

\bibitem{KCz2018}
Przemys{\l}aw Koprowski and Alfred Czoga{\l}a.
\newblock Computing with quadratic forms over number fields.
\newblock {\em J. Symbolic Comput.}, 89:129--145, 2018.

\bibitem{Koprowski2008}
Przemysław Koprowski.
\newblock Algorithms for quadratic forms.
\newblock {\em J. Symbolic Comput.}, 43(2):140--152, 2008.

\bibitem{Lam2005}
T.~Y. Lam.
\newblock {\em Introduction to quadratic forms over fields}, volume~67 of {\em
  Graduate Studies in Mathematics}.
\newblock American Mathematical Society, Providence, RI, 2005.

\bibitem{LW1992}
David~B. Leep and A.~R. Wadsworth.
\newblock The {H}asse norm theorem mod squares.
\newblock {\em J. Number Theory}, 42(3):337--348, 1992.

\bibitem{lmfdb}
The {LMFDB Collaboration}.
\newblock The {L}-functions and modular forms database.
\newblock \url{http://www.lmfdb.org}, 2021.
\newblock [Online; accessed 14 April 2021].

\bibitem{Quertier2016}
Tony Quertier.
\newblock Effective {H}asse principle for the intersection of two quadrics.
\newblock {\em LMS J. Comput. Math.}, 19(suppl. A):73--82, 2016.

\bibitem{Schicho1998}
Josef Schicho.
\newblock Rational parameterization of real algebraic surfaces.
\newblock In {\em Proceedings of the 1998 {I}nternational {S}ymposium on
  {S}ymbolic and {A}lgebraic {C}omputation ({R}ostock)}, pages 302--308. ACM,
  New York, 1998.

\bibitem{Simon2005}
Denis Simon.
\newblock Solving quadratic equations using reduced unimodular quadratic forms.
\newblock {\em Math. Comp.}, 74(251):1531--1543 (electronic), 2005.

\bibitem{Szymiczek1997}
Kazimierz Szymiczek.
\newblock {\em Bilinear algebra}, volume~7 of {\em Algebra, Logic and
  Applications}.
\newblock Gordon and Breach Science Publishers, Amsterdam, 1997.
\newblock An introduction to the algebraic theory of quadratic forms.

\bibitem{vHC2006}
Mark van Hoeij and John Cremona.
\newblock Solving conics over function fields.
\newblock {\em J. Th\'{e}or. Nombres Bordeaux}, 18(3):595--606, 2006.

\bibitem{Voight13}
John Voight.
\newblock Identifying the matrix ring: algorithms for quaternion algebras and
  quadratic forms.
\newblock In {\em Quadratic and higher degree forms}, volume~31 of {\em Dev.
  Math.}, pages 255--298. Springer, New York, 2013.

\end{thebibliography}
\end{document}